\documentclass[11pt,reqno]{amsart} \usepackage{amssymb}
\usepackage{amsthm,amsfonts}
\usepackage{xspace}
\usepackage{amsmath}
\usepackage{setspace}
\usepackage{amscd}
\usepackage{tikz}
\usepackage{enumerate}
\usepackage{graphicx}
\makeatletter
\@namedef{subjclassname@2010}{%
  \textup{2010} Mathematics Subject Classification}
\makeatother
\newtheorem{theorem}{Theorem}[section]
\newtheorem{lemma}[theorem]{Lemma}
\newtheorem{proposition}[theorem]{Proposition}
\newtheorem{cor}[theorem]{Corollary}
\newtheorem{conj}[theorem]{Conjecture}
\theoremstyle{remark}
\newtheorem{remark}{Remark}
\newtheorem{definition}{Definition}

\numberwithin{equation}{section}

\newcommand{\N}{\mbox{$\mathbb{N}$}}

\newcommand{\Z}{\mbox{$\mathbb{Z}$}}

\newcommand{\Q}{\mbox{$\mathbb{Q}$}}
\newcommand{\F}{\mbox{$\mathbb{F}$}}

\begin{document}
\title{Iterates of prime producing polynomials and  their Galois groups\vspace{-.5em}}
\author{Sushma Palimar\vspace{-.5em}}
\address{ 
Department of Mathematics,\\ 
Indian Institute of  Science,\\
Bangalore, Karnataka, India.}
\email{ sushmapalimar@gmail.com}
\begin{abstract}
\vspace{-.3em}
Let $\F_q$ be a finite field of characteristic $p>0$. 
We  prove that, given $F(t,x)\in \F_q[t][x]$ an   irreducible separable monic polynomial in the variable $x$
and  a generic monic polynomial $\phi(t)$ in the variable $t$, 
 the  polynomial $F(t,\phi)$ is a prime producing polynomial over large finite fields.  We also prove that $F(t,\phi)$ satisfies Odoni's conjecture, namely the arboreal Galois
representation associated to $F(t,\phi)$  is surjective.
\vspace{-.8em}
 \end{abstract}
\maketitle
\section{Introduction}\label{intro}
Let  $K$ be a field and
$f(x)\in K[x]$ be  a polynomial of degree $d>1$.
For any $t\in K$, the Galois groups of polynomials of the form
$f^{\circ n}(x)-t$                                                        
are called {\em arboreal} Galois groups
  where, $f^{\circ n}:=f\circ f\circ\cdots\circ f$  is the $n$-th iterate of $f$  (composition of $f$ with itself $n$ times)  [BFHJY \cite{bend}]. Given such a polynomial, its {arboreal Galois representation} is surjective if and only if the Galois group $\mathrm{Gal}(f^{\circ n}(x)-t/K)\cong  [S_{d}]^{n}$, where $[S_d]^{n}$ is 
the \textit{$n$ fold iterated wreath product of the symmetric group $S_d$}, $$[S_d]^{n}=\underbrace{S_d\wr S_d\wr\dots\wr S_d}_\text{$n$ times}=\underbrace{(S_d\wr\cdots\wr S_d)^{d}}_\text{ $(n-1)$ times}\rtimes S_d.$$  
 Let $\mathbf a: = (a_0,\dots,a_{r-1})$ be a vector of independent variables. Let $q$ be a power of $p$, $\F_q$ the field with $q$ elements and $\mathrm{char} \F_q=p>0$.  In this article we take $K=\overline{\F}_q(\mathbf a)$. Let $F(t,x)\in \overline{\F}_q[t][x]$ be any monic irreducible  separable polynomial  in $x$ of degree $\mathrm{deg}_x F=m>0$. 
The polynomial
$$\phi(t)=t^{r}+a_{r-1}t^{r-1}+\dots+a_0 \in \F_q[\mathbf a,t]
$$
is  the {\em generic} monic polynomial of degree $r$ in $t.$ 
Consider the polynomial $F(t,\phi)$ in $\overline{\F}_q[\mathbf a,t]$.
Let  $F^{\circ n}(t,\phi)$ be the $n$-th iterate of $F(t,\phi)$.  We prove:
 \begin{theorem}\label{main_th1}
 Let $F(t,x)=x^{m}+\sum_{j=0}^{m-1}g_j(t)x^{j}\in \overline{\F}_q[t][x],\mathrm{deg}_x F=m>0$ be any monic  irreducible  separable polynomial in $x$. Let $\phi(t)$ be the generic monic  polynomial in $t$ of degree $r>2$ and 
 $r>\max_{0\leq j\leq m-1}\frac{\mathrm{deg}g_j(t)}{m-j}$. Denote $d=m\cdot r$. Consider
  $F(t,\phi)\in \F_q[\mathbf a,t], \mathrm{deg}_t(F(t,\phi))=d$.
Then, 
\begin{enumerate}
\item   $\mathrm{Gal}(F(t,\phi)/\overline{\F}_q(\mathbf{a})) $ is the symmetric group $S_d$.
\item 
For $n\in \N,$  $\mathrm{Gal}(F^{\circ n}(t,\phi)/\overline{\F}_q(\mathbf a))$ is $[S_d]^{n}$. 
\item 
The arboreal Galois representation
associated to $F(t,\phi)$ is surjective. 
\end{enumerate}
\end{theorem}
\subsection*{Arboreal Galois representation}
Let  $K$ be a  field and
  $f(x)\in K[x]$ be  a polynomial of degree $d>1$. Let $f^{\circ n}$ denote the $n$-th iterate of $f$, assume that $f^{\circ n}$ are all separable for $n\geq 1$.
  For $t\in K$  the set of {\em preimages} of $t$ is the  set of roots of 
$f^{\circ n}(x)=t$ in $\overline{K}$,
\begin{displaymath}
 f^{\circ -n}(t)=\{\alpha\in \overline{K}|f^{\circ n}(\alpha)=t\}, 
\end{displaymath}
and $K(f^{\circ -n}(t))$ is a Galois extension of $K.$ 
The preimage tree associated to $f$ with root $0$ is 
\begin{displaymath}
T_{\infty}:=\bigsqcup_{n\geq 0}\{\alpha\in \overline{K}|f^{\circ n}(\alpha)=0\}.
\end{displaymath}
$T_{\infty}$ has the structure of a rooted tree in the following sense:  a  root $\alpha$ of $f^{\circ n}$ is connected to a root $\beta$ of $f^{\circ (n-1)}$ by an edge if $f(\alpha)=\beta$. 
Assume that the solutions to $f^{\circ n}(x)=0$ are distinct.
The absolute Galois group $\mathrm{Gal}(\overline{K}/K)$ acts naturally on the infinite $d$-ary rooted tree $T_{\infty}$ by its action on the $f^{\circ -n}(t)$. This yields a continuous homomorphism  $\rho:\mathrm{Gal}(\overline{K}/K)\rightarrow \mathrm{Aut}(T_{\infty})$, known as the {\em arboreal Galois representation attached to $f$}, where, $d=\mathrm{deg}f$, Jones \cite{jon}. Denote the image of $\rho$  by $\mathrm{G}_{\infty}(f)$ and note that, $$G_{\infty}=\lim_{\leftarrow}G_{n}.$$
where, $G_{n}=\mathrm{Gal}(K(f^{\circ -n}(0))/K).$
 Let $T_n$ be the complete  $d$-ary rooted tree  formed
by the first $n$ levels of $T_{\infty}$.
Clearly $G_n$ acts faithfully on $T_n$, so that $G_n\hookrightarrow \mathrm{Aut}(T_n).$
The solutions to $f^{\circ n}(x)=0$ lie in the separable closure $K_s$ of $K$ in $\overline{K}$ and $\mathrm{Gal}(K_s/K)$ acts on this copy of $T_n$. This defines a continuous homomorphism $\rho_n:\mathrm{Gal}(K_s/K)\rightarrow \mathrm{Aut}(T_{n})$ and the image of $\rho_n$ is isomorphic to $G_n$. For a more detailed description see Boston and Jones  [\cite{bj}, \cite{nr}].
A classical result of
Nekrashevych [\cite{nvk}, Proposition 1.4.2] asserts that the 
group $\mathrm{Aut}(T_{n})$ of graph automorphisms of $T_{n}$ is isomorphic to $[S_d]^{n}$.  
\begin{remark}
\begin{enumerate}
\item  Juul [\cite{jl}, Theorem 1.1], proved Theorem 1.1 in the case $F(t,x)=x.$
 \item It is interesting to note that Theorem \ref{main_th1}
holds over any algebraically closed field of characteristic $p\geq 0$ and not just over $\overline{\F}_q.$
\end{enumerate}
\end{remark}
The study of Galois groups of iterated polynomials
dates to the pivotal work of Odoni in a series of papers \cite{Odo,Odo1,Odo2}, who carried out a systematic study to know
when the representation $\rho$ is surjective and when its image
$\mathrm{G}_{\infty}$ has finite index in $\mathrm{Aut}(T_{\infty})$.
Odoni [\cite{Odo}, Theorem 1] considered the  generic polynomial $\phi$ of degree $d>1$ over the fields $K$ of characteristic zero, for which Hilbert Irreducibility Theorem (HIT) holds, (also known as, Hilbertian fields). He showed that the arboreal representation associated to $\phi$ is surjective, i.e.,
  $\mathrm{Gal}(\phi^{\circ n}/K)\cong \mathrm{Aut}(T_n)\cong [S_d]^{n}$.
 Juul [\cite{jl}, Theorem 1.1] established this result for generic polynomials of degree $d$ defined over an algebraically closed field  of characteristic $p>0$ except in the case $p=d=2$. 
 Examples of Hilbertian fields include $\Q$,
number fields, and finite extensions of $K(t)$ for any field $K$, (see Fried \& Jarden 
[\cite{jm}, Chapter 12]).
 The following conjecture was proposed by Odoni.
 \begin{conj}Odoni [\cite{Odo}, Conjecture 7.5] \label{conj1}For any Hilbertian field $K$ of characteristic $0$ and any $d\geq 2,$ there is a monic polynomial of degree $d$ such that $\mathrm{Gal}(K(f^{\circ -n}(0))/K)\cong [S_d]^{n}$
 for all $n\geq 0.$
 \end{conj}
 
 For any point $t\in K$, if we set $g(x)=f(x+t)-t\in K[x]$  then the fields $K(f^{\circ -n}(t))$ and $K(g^{\circ -n}(0))$ coincide. Thus, one can state Odoni's conjecture in terms of the preimages $f^{\circ -n}(t)$ of an arbitrary $K$-rational point $t$ instead of $0$. (See Benedetto and Juul \cite{bj1}).
 The Conjecture \ref{conj1} asserts that  the associated arboreal representation $\rho: \mathrm{Gal}(\overline{K}/K)\rightarrow\mathrm{Aut}(T_{\infty})$ is surjective.
 
 Odoni proved  Conjecture \ref{conj1}  over $K=\Q$ for the quadratic  polynomial $x^2-x+1$. 
Looper \cite{nc}, proved Conjecture \ref{conj1} over the field $K=\Q$, for all polynomials of prime degrees $d$.
Looper's result was independently extended in [Kadets \cite{bk}, Benedetto \& Juul \cite{bj1} and Specter \cite{js1}]. 
Kadets \cite{bk}, proved
Conjecture \ref{conj1} over $\Q$ for polynomials of even degree $d \geq 20$.
Specter  \cite{js1} proved Odoni's conjecture for all number fields,  more generally, for all algebraic extensions $K/\Q$ that are
unramified outside of a finite set of primes.
Benedetto and Juul  \cite{bj1} proved Odoni's
conjecture for even degree polynomials over an arbitrary number field, and for odd degree
polynomials over number fields $\F$ that do not contain $\Q(\sqrt{d},\sqrt{d + 1})$.
In Theorem \ref{main_th1} we prove Odoni's conjecture
over the field $\F_q(\mathbf a)$, where $(\mathbf a)=(a_0,\dots, a_{r-1})$ is an $r$-tuple of indeterminates defined over $\F_q$.
We consider $F(t,x)=x^{m}+\sum_{j=0}^{m-1}g_j(t)x^j\in \overline{\F}_q[t][x],\mathrm{deg}_x=m>0$, a monic  irreducible  polynomial in $x$, which is separable  over $\overline{\F}_q(t)$.

By virtue of HIT, the irreducibility of $F(t,x)$ in $x$ 
implies, $\F_q$ contains infinitely many elements $s$ such that $F(t,s)$ is irreducible in $\F_q[t]$. 
This further implies, there exists monic irreducible polynomials $f(t)\in \F_q[t]$ of bounded degree such that, $F(t,f(t))$ is irreducible in $\F_q[t]$. 
In order  to study the general case, we evaluate $x$  by 
the generic polynomial 
$\phi(t)=t^{r}+a_{r-1}t^{r-1}+\dots+a_0\in \F_q[\mathbf a,t]$ of degree $r$ in $t$, such that $r>2$ and $r>\max_{0\leq j\leq m-1}\frac{\mathrm{deg}g_j(t)}{m-j}$.   The polynomial  $F(t,\phi)\in \overline{\F}_q[\mathbf a,t]$ is  monic in $t$ of degree $d=m\cdot r$. We show that, the arboreal representation associated to $F(t,\phi)$ is surjective.
For $n\in \N$, we denote by $K_n$, the splitting field of 
$F^{\circ n}(t,\phi)$ over $\overline{\F}_q(\mathbf a)$.
In order  to prove Theorem \ref{main_th1}, we calculate
the Galois group of $F(t,\phi)$ over $\overline{\F}_q(\mathbf a)$ and invoke the property  of linear independence of Galois extensions to
show that
$\mathrm{Gal}(K_{n}/K_{n-1})=(S_d)^{d^{n-1}}$,  the direct product of $d^{n-1}$ symmetric groups $S_d$  for all $n\geq 1$.
\subsection{Techniques}
The proof of  $\mathrm{Gal}(K_n/K_{n-1})=(S_d)^{d^{n-1}}$ relies on a classical result 
 of algebra known as,  Capelli's Lemma, which we will use many times
throughout the paper. We state it below without proof.
\begin{lemma}(Capelli's Lemma)\label{capl}
 Let $K$ be any field and let $f, g \in K[x]$.
Suppose $\alpha \in \overline {K}$ is any root of $f$. Then $f(g(x))$ is irreducible over $K$ if and
only if both $f(x)$ is irreducible over $K$ and $g(x)-\alpha$ is irreducible over $K(\alpha)$.
\end{lemma}
For a fixed $n\in \N$, the $n$-th iterate of $F(t,\phi)$ is  defined as:
\begin{displaymath}
F^{\circ n}:=F^{\circ n-1}(F(t,\phi))
\end{displaymath}
 By our notation, the splitting field of  $F^{\circ n-1}(t,\phi)$ is denoted as $K_{n-1}$.
Let $\beta_1,\dots,\beta_{d^{n-1}}$ be the $d^{n-1}$ distinct roots of $F^{\circ n-1}(t,\phi)$ in some algebraic closure of $\overline{\F}_q(\mathbf a)$; so that, $K_{n-1}=\overline{\F}_q(\mathbf a,\beta_1,\dots,\beta_{d^{n-1}})$. 
By definition,
\begin{equation}
 F^{\circ n}(t,\phi)=F^{\circ n-1}(F(t,\phi))=\prod_{i=1}^{d^{n-1}}F(t,\phi)-\beta_i
 \end{equation}
  $F(t,\phi)$ is irreducible in $t$ and separable over $\overline{\F}_q(\mathbf a)$ (given in  Proposition \ref{gal}). Since  
 each of 
 $F(t,\phi)-\beta_i$ is linear in $\beta_i$,   $F(t,\phi)-\beta_i$ is irreducible  and separable over $\overline{\F}_q(\mathbf  a, \beta_i)$, for $1\leq i\leq d^{n-1}$.
 Thus by  Capelli's lemma $F^{\circ n}(t,\phi)$ is irreducible in $t$ and separable over $\F_q(\mathbf a)$.
 To each  \textit{shifted polynomial} $F(t,\phi)-\beta_i$ we can associate a field extension $M_i/\overline{\F}_q(\mathbf a,\beta_i)$.
Our objectives in Theorem \ref{main_th1} is to
show that the Galois group of $F(t,\phi)$ over $\overline{\F}_q(\mathbf a)$ is $S_d$ and 
$M_1,\dots,M_{d^{n-1}}$ are linearly disjoint $S_d$-Galois extensions of $\overline{\F}_q(\mathbf a,\beta_1),\dots,\overline{\F}_q(\mathbf a,\beta_{d^{n-1}})$ respectively. Equivalently, 
\begin{equation}\label{li_gal}
 \mathrm{Gal}\Big(\prod_{i=1}^{d^{n-1}}F(t,\phi)-\beta_i/\overline{\F}_q(\mathbf a,\beta_i)\Big)\cong (S_d)^{d^{n-1}}.
\end{equation}
which is equivalent to the surjectivity of the permutation sign map (see Proposition \ref{se_iso} and Remark \ref{rem_iso}),
 \begin{equation}\label{sgn}
  \mathrm{Gal}(K_n/K_{n-1})\rightarrow \underbrace{(S_{d}/A_{d})\times\dots\times(S_{d}/A_{d})}_\text{ $d^{n-1}$ times }\cong \{\pm 1\}^{d^{n-1}}.
 \end{equation}
 These methods on linear independence serve as an alternate proof to Juul [\cite{jl}, Theorem 1.1].
 \subsection{Related   asymptotics}
 We conclude the introduction with a note on  special  application associated of Theorem \ref{main_th1}.  In \S \ref{Gal_1}, we show that $\mathrm{Gal}(F(t,\phi)/\overline{\F}_q(\mathbf a))=S_d$.
 Therefore by HIT
 one can specialize $\phi\rightarrow f\in \F_q[t]$ by taking $\F_q$ values for $a_i$ such that   $F(t,f)$ is irreducible over $\F_q$.
 By the work of Cohen [\cite{sdc3},
 Theorem 3], one can find
 the number of such monic $f\in \F_q[t]$ of degree $r$ in $t$, for which $F(t,f)$ is irreducible in $\F_q[t]$.
 The following result is  a Corollary  of 
 [Theorem 3, \cite{sdc3}] and Theorem \ref{main_th1}.
\begin{cor}\label{cor}
Assume the degree $\mathrm{deg}_t(F(t,f))=d$ is fixed. Then the number of  monic polynomials $f(t)\in \F_q[t]$  of degree  $r>\max_{0\leq j\leq m-1} \frac{\mathrm{deg}g_j(t)}{m-j}$ and $r>2$ for which $F(t,f(t))$ is irreducible in $\F_q[t]$ is
\begin{displaymath}
 \frac{q^{r}}{d}+O_{d}(q^{r-\frac{1}{2}}), \quad q\rightarrow \infty
\end{displaymath}
 the implied constant depends only on $d$, degree $\mathrm{deg}_t(F(t,f))$.
 \end{cor}
 \begin{remark}
  In 1966
Sarvadaman Chowla conjectured that the
 number of polynomials  of the form
\begin{equation}\label{lab0}
x^n+x+d
\end{equation}
which are irreducible modulo $p$ is asymptotic to $\frac{p}{n}$, as $p\rightarrow \infty$ for fixed $n.$ This was proved independently by  Cohen \cite{sdc} in 1970 and  Ree \cite{rr} in 1971, that,  for any prime power $q$, the number of $d$ such that $x^n+x+d$ is irreducible is indeed $\frac{q}{n}+O(q^{\frac{1}{2}})$  with
the implied constant depending only on $n$.
 One of the ingredient in the proof is the fact that the Galois group of the 
polynomial $x^n+x+t\in \F_q[t,x]$ over the function field  $\overline{\F}_q(t)$ is the 
symmetric group $S_n$ of order $n.$ 
The connection of these problems to Galois theory is established through  an explicit form of the Chebotarev density theorem for function fields,  which in turn depends on Weil's classical result on the Riemann hypothesis for curves over finite fields.
 Later Cohen [\cite{sdc3}, Theorem 3]  considered more general monic polynomials of degree $n$ in $\F_q[x]$ of the form:
\begin{equation}\label{coh}
 \{x^{n}+g_0(x)+a_1g_1(x)+\dots+a_mg_m(x):(a_1,\dots,a_m)\in \F_q^{m}\}
\end{equation}
where $g_i\in \F_q[x]$ has degree less than $n$ for $0\leq i\leq m$, and $g_1,\dots,g_m$ are linearly independent over $\F_q$. He proved that, as $q\rightarrow\infty$, the number of points $(a_1,\dots,a_m)\in \F_q$ for which the form  (\ref{coh}) is irreducible   in $\F_q[x]$ is  
$$\frac{q^{m}}{n}+O(q^{m-1/2}),$$ where the implied constant depends only on $n$.
Using the fact that the corresponding Galois group is $S_n$,
the required asymptotic is 
derived 
through an explicit form of the Chebotarev density theorem
for function fields. (See 
Gao, Howell and Panorio \cite{gao} for more details on  Cohen's results).
 \end{remark}

\subsection{Stable polynomials}\label{st}
A polynomial $f\in K[t]$ of degree at least $2$ is said to be {\em stable} if $f^{\circ n}$ is irreducible over $K$ for all $n\geq 1.$
 Jones and Boston [\cite{rj1}, Proposition 2.3], proved that if $f$ is a quadratic polynomial over a finite field $\F_q$ of odd characteristic then $f$ is stable
if and only if 
the set of {\em adjusted critical orbit of $f$, } \begin{displaymath}
     \mathrm{Orb}_{adj}(f) =   \{-f(\gamma)\}\cup\{f^{\circ i}(\gamma):i=2,3,\dots,\}
        \end{displaymath}
does not contain squares,  where $\gamma$ is the unique critical point of $f$.
When $K=\F_q$ is a finite field, there is some $k$ such that $f^{\circ k}=f^{\circ l}$ for $l<k$ and accordingly $\# \mathrm{Orb}_{adj}(f)=l.$
       
Perez, Nicolas, Ostafe and Sadornil  [\cite{prez}, Theorem 5.5] found the number of stable polynomials of degree $d\geq 2$ over finite fields $\F_q$ of characteristic $p\neq 2$. 
By Theorem \ref{main_th1},  $F^{\circ n}(t,\phi)$ is irreducible and separable over $\overline{\F}_q(\mathbf a)$ for all $n\geq 1.$ 
By HIT, there exists irreducible specialization for $\phi(t)\rightarrow f(t)\in \F_q[t]$ such that,
$F(t,f)$,$\dots$,$F^{\circ n}(t,f)$,$\dots$, are irreducible in $\F_q[t]$. Hence by definition,
the polynomial $F(t,f)$ is stable over $\F_q$. 
By the work of 
Andrade, Bary-Soroker and Rudnik [\cite{abz}, Theorem 3.1],
one can 
find the number of polynomials $f(t)\in \F_q[t]$ such that $F(t,f)$ is stable over finite fields $\F_q$ of characteristic $p>0$, as $q\rightarrow\infty.$ Under the assumptions of Theorem \ref{main_th1},
suppose the set  of adjusted critical orbit 
$\mathrm{Orb}_{adj}(F(t,f))=\{F(t,f),F^{\circ 2}(t,f),\dots,F^{\circ l}(t,f)\}$ be of cardinality $l>1$.
We note that each of 
$F(t,f),F^{\circ 2}(t,f),\dots,F^{\circ l}(t,f)$ is distinct in $\F_q[t]$ and   separable over $\overline{\F}_q$, which is nothing but, 
$\mu(F^{\circ i}(t,f))=\pm 1$ for $i=1,\dots,l$, where $\mu$ denotes the M\"{o}bius function for $\F_q[t]$. Thus, we have the following result which is a Corollary to Theorem \ref{main_th1} and [\cite{abz}, Theorem 1.4].
\begin{cor}\label{stability}
 Let $l>1$ be an integer and $\mathrm{deg}_t(F^{\circ i}(t,\phi))$ be fixed for $1\leq i\leq l$. Then, number of  monic  polynomials $f(t)\in \F_q[t]$ of degree $r>2$ and $r>\max_{0\leq j\leq m-1}\frac{\mathrm{deg}g_j(t)}{m-j}$  for which 
 $ F(t,f)$ is stable over $\F_q$ such that $\#\mathrm{Orb}_{adj}(F(t,f))=l$ is
 \begin{equation}\label{stable}
  \frac{q^{r}}{2^{l}}+O_{l,\mathrm{deg}F^{\circ i}}(q^{r{-1}/{2}}),\quad q\rightarrow \infty
 \end{equation}
where the implicit constant in the $O$-notation depends only on $l$ and the degrees $\mathrm{deg}_tF^{\circ i}$ for $1\leq i\leq l.$
 \end{cor}
 Corollary \ref{stability} follows  from Proposition \ref{arb} on
 the linear independence of the discriminants of 
$F^{\circ l}(t,\phi)$, $1\leq i\leq l$,   in the $\F_2$ vector space $\overline{\F}_q(\mathbf a)^{\times}/\overline{\F}_q(\mathbf a)^{\times 2}$.  
(The discriminant of a monic separable polynomial  over a field $K$ is defined in \S \ref{sqind}).
The required estimate is obtained from the function field analog of Chebotarev density theorem \cite{abz}. 
  \begin{remark}
   Perez, Nicolas, Ostafe and Sadornil [\cite{prez}, Theorem 5.5] proved Corollary \ref{stability} in the case of $F(t,f)=f$ over finite fields $\F_q$ of characteristic $p\neq 2.$
 \end{remark}
 Corollary \ref{stability} relates to the polynomial analog of Chowla conjecture on the autocorrelation of M\"{o}bius $\mu$ function for polynomials in $\F_q[t]$.
   This result was resolved by Carmon and Rudnick  [\cite{cdr}, Theorem 1.1]    when $q$
is odd, and  by Carmon \cite{cdr}   when $q$ is even.
We prove the first part of Theorem \ref{main_th1}  in \S \ref{Gal_1} and, rest of the proof  in \S \ref{iter}. In \S \ref{appl} we give a brief sketch of proof of Corollary  \ref{stability}. 

\subsection{Related results}
 The linear independence of $S_d$-Galois extension  plays an important role in establishing the polynomial analog (in the ring $\F_q[t]$) of classical conjectures in number theory  stated over $\Z.$

 For example, the work of Cohen [Theorem 3, \cite{sdc3}] led to the 
 study of
 general theorem on \textit{auto-correlations
of arithmetic class functions} over large finite fields;  these include characteristic functions on primes  (polynomial analog of Hardy-Littlewood prime tuple conjecture)  the M\"{o}bius $\mu$ function  and divisor functions (e.g.,    
 the Titchmarsh divisor problem). The required asymptotic is  established via the Chebotarev Density Theorem by considering the fact that, the  Galois group of the product of $m$ polynomials of degree $n>1$ is $(S_n)^{m}$, the direct product of $m$ copies of the symmetric group $S_n$.  These results are studied
in  great detail in the work of 
 Pollack [\cite{pol1}, \S 2.2, Lemma 5], Bary-Soroker [\cite{lbs1}, Proposition 1.7] and  ABR [\cite{abz}, Theorem 1.4].

 Furthermore, Schinzel Hypothesis H is a  conjecture in number theory which asserts that
 finitely many irreducible polynomials over integers simultaneously assume prime values infinitely often unless there is a local obstruction. The analogy between the rings $\Z$ and $\F_q[t]$,  suggests a
naive Hypothesis H analog for the ring $\F_q[t]$, that allows us to produce irreducible polynomials satisfying
certain properties.
 The quantitative  analog of this hypothesis H, is the Bateman-Horn conjecture [Bender \& Wittenberg \cite{aob}]. 
The polynomial analog of  Bateman-Horn conjecture over finite fields of characteristic $p>0$  is to find the asymptotic for  the simultaneous irreducibility of finite number of polynomials of bounded degree in  $\F_q[t]$.
Pollack [\cite{pol1}, Theorem 2] resolved
 the polynomial analog of the Schinzel Hypothesis by establishing the linear independence of $S_d$-Galois extensions  over $\F_q(t)$ and its arithmetic  analog, the Bateman-Horn conjecture over $\F_q$ for $\mathrm{char}\F_q\neq 2$.
Bary-Soroker \cite{lbs1} studied the polynomial analog of the Schinzel Hypothesis  over an algebraically closed field of characteristic $p\geq 0.$

In future we expect to relate these results to the minimum ramification problem.
The minimal ramification problem asks for the minimal number 
of ramified primes $r_{K}(G)$ in Galois extensions $L/K$ whose Galois group is a given group $G.$
Following the results by Meghan \cite{meg} for $K=\F_p(t)$, Bary-Soroker, Entin and Fehm [\cite{bef}]
proposed the following
 conjecture.  
 \begin{conj}\label{bef}(\cite{bef}, Conjecture 1.4)
  Let $q=p^{\nu}$ and  $K=\F_q(t)$. Then
  for a nontrivial finite group $G$ in a Galois extension $L/K$,
the minimal number of ramified primes $r_{\F_q(t)}(G)$ is 
          $\max\{1,d\} $.
 \end{conj}
Here, $d=d(G/p(G))^{ab}$ denotes the minimal number
of elements needed to generate  $G^{ab}$ and $p(G^{ab})$ is the subgroup of $G^{ab}$ generated by the $p$-Sylow subgroups of $G^{ab}$. 
Meghan [\cite{meg}, Theorem 2.6] and BEF  [\cite{bef}, Theorem 3.3] proved Conjecture \ref{bef} for nontrivial finite abelian group $G$.
Thus, Conjecture \ref{bef} predicts that,  there
exists an $S_n$-extension of $\F_q(t)$ ramified at only one prime, i.e., $r_{\F_q(t)}(S_n)=1$.
Additionally Meghan \cite{meg} observes that Conjecture \ref{bef} holds when $G$ is the direct product of finite copies of $S_n$, the symmetric group of order $n.$ This is proved in [\cite{bef}, Theorem 1.8]  under specific conditions on $n$ and $p$.
Therefore,  we may expect that Conjecture \ref{bef} holds in the following case. 
Let $x,t$ be indeterminates over $\F_q$ and $f(x)\in \F_q[x],\mathrm{deg}f(x)=d>0$ such that $\mathrm{Gal}(f^{\circ n}-t/\overline{\F}_q(t))=[S_d]^{n}$. Then, under certain conditions on $\mathrm{char}(\F_q)$ and the degree $d$, Conjecture \ref{bef} holds for $\mathrm{Gal}(f^{\circ n}-t/\overline{\F}_q(t))$. We plan to study this in our future work.

 \section{ Evaluating $\mathrm{Gal}(F(t,\phi)/\overline{\F}_q(\mathbf a))$}\label{Gal_1}
  In this section we prove  that $\mathrm{Gal}(F(t,\phi)/\overline{\F}_q(\mathbf a))\cong S_d$,   by showing that
 $\mathrm{Gal}(F(t,\phi)/\overline{\F}_q(\mathbf a))$ is a doubly transitive subgroup of $S_d$  containing a transposition.
 \begin{proposition}\label{gal}
The polynomial $F(t,\phi)$ in 
$\overline{\F}_q[\mathbf a,t]$ is an  irreducible polynomial which is  separable  over $\overline{\F}_q(\mathbf a)$ and monic  in $t$ of degree $d$.
\end{proposition}
\begin{proof}
Let us recall, $F(t,x)=x^{m}+\sum_{j=0}^{m-1}{g_j(t)x^{j}}\in {\F}_q[t][x]$ is a monic irreducible and separable polynomial  in $x$ of degree $m>0$.
The Hilbert irreducibility theorem guarantees 
 the existence of infinitely many monic polynomials $g(t)\in\F_{q}[t]$ of bounded degree for which the polynomial $F(t,g(t))$ is  irreducible in  $\F_q[t]$.  Thus for the generic polynomial $\phi(t)=t^{r}+a_{r-1}t^{r-1}+\dots+a_0$, defined over $\F_q$ of degree $r$ in $t$, such that $r>\max_{0\leq j\leq m-1}\frac{\mathrm{deg}g_j(t)}{m-j}$ and $r>2$   the polynomial
 $F(t,\phi)$ is irreducible in $\F_q[t]$ of degree $d=m\cdot r$ in $t$. 
 
 Uchida   \cite{uch} proved that any Hilbertian field is separably Hilbertian, i.e, if a   polynomial $F(t,x)$,  is  separable over $\overline{\F}_q(t)$, then there exists infinitely many specializations $s\in \F_q$ such that $F(t,s)$ is separable over $\F_q$. The work of  Rudnik [\cite{zr}, Theorem 1.1] asserts that,
 if $F(t,x)$ is a separable polynomial over $\F_q(t)$, then
 for almost all monic polynomials $g(t)\in \F_q(t)$ of bounded degree,  $F(t,g(t))$ is square-free  in any extension of $\F_q(t)$. Therefore, for the  generic monic polynomial 
 $\phi(t)$,
the  polynomial
 $F(t,\phi)$ is separable over $\overline{\F}_q(\mathbf a)$.  
Thus,  $\mathrm{Gal}(F(t,\phi)/\overline{\F}_q(\mathbf a))$ is a transitive subgroup of $S_d$, the symmetric group on  $d$ letters.
\end{proof}
 The notion of Hasse-Schmidt derivative comes as a replacement for the usual derivative in a field of positive characteristic.
Suppose $K$ is any algebraically closed field of characteristic $p\geq 0.$ 
Let $f=\sum_{i=0}^{n} a_ix^i\in K[x].$
Define, 
\begin{displaymath}
D^{[k]}f:=\sum_{i=k}^{n}  \binom {i}{k} a_i x^{i-k} \in K[x]
\end{displaymath}
       The function $D^{[k]}f:K[x]\rightarrow K[x]$ is called the $k$-th Hasse  derivative; $D^{[1]}f$ is the usual derivative $f^{\prime}$ of $f$.
 \begin{proposition}\label{trans-p}
$\mathrm{Gal}(F(t,\phi)/\overline{\F}_q(\mathbf{a}))$   contains a transposition. 
\end{proposition}
\begin{proof}
Separability of $F(t,\phi)$ over $\overline{\F}_q(\mathbf a)$ implies $F(t,\phi)$  has  double roots in any extension of $\overline{\F}_q(\mathbf a)$.  
The only primes ramifying in the splitting field of $F(t,\phi) $ over $\overline{\F}_q(\mathbf a)$ are the primes dividing the discriminant $\mathrm{Disc}(F(t,\phi))$ of $F(t,\phi)$. Let $\F$ be the splitting field of $F(t,\phi)$ over $\overline{\F}_q(\mathbf a)$,
and $P\neq 0$ be any prime  dividing the discriminant of $F(t,\phi)$. Clearly, $P$ is a ramified prime in  $\F.$
Since $(\mathbf a)=(a_0,\dots, a_{r-1})$ are algebraically independent variables over $\F_q$,
 $F^{\prime}(t,\phi)$ and $D^{[2]}F(t,\phi)$ do not have any common roots.
i.e., $F^{\prime}(t,\phi)$ is separable over $\overline{\F}_q(\mathbf a)$. Therefore,  for any prime $P$ of $\F_q(\mathbf a)$,  every irreducible factor of 
  $F(t,\phi) \bmod{P}$ 
has ramification index less than $3$,
hence, there can be at most
one double root $\pmod {P}$ for each ramified prime $P$. This shows that the inertia
subgroup at $P$ is either trivial, or is of order two generated by a transposition. The inertia group is tame for  $p\neq 2$ and wild for $p=2$ and in both the cases inertia group is generated by a transposition. 
 \end{proof}
\begin{proposition}\label{prim-do}
  $\mathrm{Gal}(F(t,\phi) /\overline{\F}_q(\mathbf{a}))$ is doubly transitive.
\end{proposition}
\begin{proof}
To
 prove the  double transitivity of $\mathrm{Gal}(F(t,\phi) 
/\overline{\F}_q(\mathbf a))$
we follow the work of  Abhyankar [\cite{ssa1}, \S 4], known as {\em the method of throwing away roots}.
 Let  $\alpha_1,\dots,\alpha_d$ be the $d$  distinct roots of $F(t,\phi)$ in its splitting field $\F$. $F(t,\phi)$ splits into $d$ distinct linear factors $(t-\alpha_1)\dots(t-\alpha_d)$ in $\F[t]$.  
By eliminating a root $t=\alpha_1$ of $F(t,\phi)$, we  get,
\begin{equation}
 F_1(\mathbf a,t)=\frac{F(t,\phi)}{(t-\alpha_1)}\in \overline{\F}_q(\mathbf a)(\alpha_1)[t]
\end{equation}
Each of the  roots
 $\alpha_i$ for $1\leq i\leq d$ are transcendental over $\overline{\F}_q(\mathbf a)$.
Obviously, $F_1(\mathbf a,t)$ is irreducible  in $\overline{\F}_q(\mathbf a)(\alpha_1)[t]$ and separable over $\overline{\F}_q(\mathbf a)(\alpha_1)$. 
Consequently, $F(t,\phi)$ and $F_1(\mathbf a,t)$ are irreducible in $\overline{\F}_q(\mathbf a)[t]$ and $\overline{\F}_q(\mathbf a)(\alpha_1)[t]$ respectively.
This means, the stabilizer of the root 
$t=\alpha_1$ in the Galois group of $F_1(\mathbf a,t)$ acts transitively on other roots,
 implying the Galois group  of $F(t,\phi)$ over $\overline{\F}_q(\mathbf{a})$ is doubly transitive.
Since any doubly transitive group action is primitive, $\mathrm{Gal}(F(t,\phi)/\overline{\F}_q(\mathbf a))$ is primitive.
\end{proof}
By the fact that any primitive permutation group containing a transposition is symmetric, 
(see Serre [\cite{JP}, Lemma 4.4.3]), $\mathrm{Gal}(F(t,\phi)/\overline{\F}_q(\mathbf a))$ is isomorphic to $S_d$.
This concludes the proof  (1) of Theorem 1.1.
\section{Galois groups of iterates of $F(t,\phi)$ over $\overline{\F}_q(\mathbf a)$.} \label{iter}
 In this section we prove condition (2) and (3) of Theorem \ref{main_th1}.
 Let us recall that for a fixed $n\in \N$,   the $n$th iterate of $F(t,\phi)$ is denoted as,
 \begin{displaymath}
F^{\circ n} (t,\phi):=F^{\circ n-1} (F(t,\phi))
\end{displaymath}
 and let $K_n$ be  the splitting field of $F^{\circ n}(t,\phi)$ over $\overline{\F}_q(\mathbf a)$
 which is the constant field extension of $\overline{\F}_q(\mathbf{a})$ contained in a fixed algebraic closure $\Omega$ of $\overline{\F}_q(\mathbf{a})$. 
 We set, $F^{\circ 1}(t,\phi)$ as $F(t,\phi)$ and $K_0=\overline{\F}_q(\mathbf a)$.
 The objective of the second and third part of Theorem \ref{main_th1} is to show that for $n\in \N$, $$\mathrm{Gal}(F^{\circ n}(t,\phi)/\overline{\F}_q(\mathbf a))=(S_d)\times(S_d)^{d}\times\cdots \times(S_d)^{d^{n-1}}.$$
\begin{proposition}\label{lin_sep}
 With the above notation, for any positive integer $n$, 
 $F^{\circ n}(t,\phi)\in \overline{\F}_q(\mathbf a)[t]$ is  a monic irreducible polynomial  of degree $d^{n}$ in $t$ which is  separable over $\overline{\F}_q(\mathbf a)$.
\end{proposition}
\begin{proof}
 We prove the statement using induction on $n$.
 For $n=1$, the polynomial $F(t,\phi)$ is irreducible in $t$ and separable over $\overline{\F}_q(\mathbf a)$.
Therefore the result holds true for $n=1$. For $n\geq 2$ we use Capelli's Lemma.
Suppose  the result holds true for  $n-1$. Let $\beta_1,\dots,\beta_{d^{n-1}}$ be the $d^{n-1}$ distinct roots of $F^{\circ n-1}(t,\phi)$ in its splitting field $K_{n-1}/\overline{\F}_q(\mathbf a)$. 
Then,
\begin{equation}\label{kn+1prod}
 F^{\circ n}(t,\phi)=F^{\circ n-1}(F(t,\phi))=\prod_{i=1}^{d^{n-1}}F(t,\phi)-\beta_i
 \end{equation}
 The $d^{n-1}$ distinct factors, $F(t,\phi)-\beta_i$ on the right of (\ref{kn+1prod}) are  linear in $\beta_i$, therefore, these polynomials are irreducible  and separable over $\overline{\F}_q(\mathbf a, \beta_i)$ for $1\leq i\leq d^{n-1}$.
 Additionally, $F(t,\phi)-\beta_i\in K_{n-1}[t], \mathrm{deg}_t(F(t,\phi)-\beta_i)=d$
  are irreducible and pairwise-coprime  in $K_{n-1}[t]$.
Thus, $F^{\circ n}(t,\phi)$ is an irreducible and separable polynomial over $\overline{\F}_q[\mathbf a]$ which is monic in $t$ of degree $d^{n}$.
 \end{proof}
Denote  by $M_i$, the  splitting field of 
 $F(t,\phi)-\beta_i$ over $\overline{\F}_q(\mathbf a,\beta_i)$ for $1\leq i\leq d^{n-1}$.
 Then since the $d^{n-1}$ distinct roots $\beta_1,\dots,\beta_{d^{n-1}}$   of $F^{\circ n-1}(t,\phi)$ are  transcendental over 
$\overline{\F}_q$, 
\begin{equation}
   \mathrm{Gal}(F(t,\phi)/\overline{\F}_q(\mathbf a)) \cong  \mathrm{Gal}(F(t,\phi)-\beta_i/\overline{\F}_q(\mathbf a,\beta_i))\cong S_d
  \text{ for } 1\leq i\leq d^{n-1}.
  \end{equation}
  (See for example, Pollack [\cite{pol1}, \S 2.2]).
  Equivalently,
  \begin{displaymath}
 \mathrm{Gal}(M_i/\overline{\F}_q(\mathbf a,\beta_i))\cong S_d \text{ for } 1\leq i\leq d^{n-1}.
\end{displaymath}
and  $$M_i\cap M_j=\overline{\F}_q(\mathbf a), \text{ for }1\leq i\neq j \leq d^{n-1}.$$
Therefore the Galois extensions,  
\begin{equation}\label{disj_spli}
M_1/\overline{\F}_q(\mathbf a,\beta_1), \dots, M_{d^{n-1}}/\overline{\F}_q(\mathbf a,\beta_{d^{n-1}})
 \end{equation}                                                                             
are linearly disjoint over $\overline{\F}_q(\mathbf a)$, so that, 
\begin{displaymath}
 \mathrm{Gal}\Big(\prod_{i=1}^{d^{n-1}}F(t,\phi)-\beta_i/\overline{\F}_q(\mathbf a,\beta_i)\Big)\cong (S_d)^{d^{n-1}}.
\end{displaymath}
 (See, [Lang [\cite{lan}, VI, \S 1.14, \S 1.15] for example).

\begin{proposition}\label{se_iso}
Denote $M_i^{\prime}$:=$M_iK_{n-1}$,  for 
$ 1\leq i\leq d^{n-1}.$
 Then $M_i^{\prime}/K_{n-1}$ is a Galois extension and the  Galois group $
 \mathrm{Gal}(M_i^{\prime}/K_{n-1})\cong S_d$. Moreover, $M_1^{\prime}/K_{n-1}$, $\dots$, $M_{d^{n-1}}^{\prime}/K_{n-1}$ are linearly disjoint over $K_{n-1}$.
\end{proposition}
\begin{proof}
From the above discussion,
each of $M_i/\overline{\F}_q(\mathbf a,\beta_i)$ is an $S_d$-Galois extension.
Changing the base by $K_{n-1},$ 
the extension $M_i^{\prime}/K_{n-1}$ is a Galois extension and $M_{i}^{\prime}\cap M_{j}^{\prime} =K_{n-1}$  for 
$1\leq i\neq j\leq d^{n-1}.$
Therefore,
$M_{1}^{\prime},\dots,M_{d^{n-1}}^{\prime}$ are linearly disjoint over $K_{n-1}$. 
Thus, the field extension $K_{n}/K_{n-1}$ is the compositum of linearly disjoint field extensions $M_{1}^{\prime}/K_{n-1},\dots,M_{d^{n-1}}^{\prime}/K_{n-1}$ written as, 
\begin{equation}\label{lin-hol}
K_{n}/K_{n-1}= M_{1}^{\prime}/K_{n-1}\cdots M_{d^{n-1}}^{\prime}/K_{n-1}. 
\end{equation}
Furthermore, 
$\mathrm{Gal}(M_iK_{n-1}/K_{n-1})$ is a normal subgroup of $ \mathrm{Gal}(M_i/\overline{\F}_q(\mathbf a,\beta_i))$,  and, $$M_i\cap K_{n-1}=\overline{\F}_q(\mathbf a,\beta_i).$$ Therefore,
\begin{displaymath}
\mathrm{Gal}(M_iK_{n-1}/K_{n-1})\cong \mathrm{Gal}(M_i/M_i\cap K_{n-1})=\mathrm{Gal}(M_i /\overline{\F}_q(\mathbf a,\beta_i))=S_d.
  \end{displaymath}
  (See Ash [\cite{rba}, Theorem 6.2.2] for example).
  Thus,
  \begin{equation}\label{la1-eq}
  \mathrm{Gal}(M_iK_{n-1}/K_{n-1})
= \mathrm{Gal}(M_i^{\prime}/K_{n-1})= S_d \text{ for } 1\leq i\leq d^{n-1}.
\end{equation}
\end{proof}
\subsection{Square-free  discriminants}\label{sqind}
We begin by recalling some basic facts about the discriminant of a polynomial.
 Suppose $f(x)$ is a  monic separable polynomial of degree $n$ over an arbitrary field $\F$ and
 $\rho_1$,$\dots$,$\rho_n$ are  the roots of $f(x)$ in some extension field $L$ of  $\F$.
  The quantity 
 \begin{equation}\label{delta}
 \delta(f)=\prod_{i<j}(\rho_i-\rho_j).\end{equation}
 satisfies a quadratic equation, 
 $\mathrm{Disc}{f}=\delta^{2}$ 
 where  $\mathrm{Disc}f$ is the discriminant of $f$.
 In the case of odd characteristic $p$, $\delta(f)$ is preserved by 
  the even permutations of the roots of $f(x)$.
  If   $\mathrm{Disc}(f)$ has no square root in $\F$ then, 
  $\mathrm{Disc}(f)$ can be thought of as an element in $ \F^{*}/\F^{*2}$, where $\F^*=\F-\{0\}$ (see Wadsworth \cite{aw}).
  In the even characteristic case,  
 the discriminant of a monic separable polynomial is defined in terms of 
Berlekamp's  polynomial discriminant.
\begin{definition} \label{ber}
Suppose $f(x)$ is a  monic separable polynomial of degree $n$ over an arbitrary field $\F$ of even characteristic.
 Let $\rho_1,\dots,\rho_n$ be the $n$ distinct roots of $f$ in a splitting field over $\F$. The Berlekamp's discriminant of $f$ is given by
\begin{displaymath}
 \mathrm{Berl}(f)=\sum_{1\leq i\leq n}\frac{\rho_i\rho_j}{\rho_i^{2}+\rho_j^{2}}.
\end{displaymath}
Let
\begin{equation}\label{beta}
 c=\sum_{1\leq i\leq j\leq n}\frac{\rho_i}{\rho_i+\rho_j},
\end{equation}
then, $c$ satisfies the quadratic equation $c^2+c=\mathrm{Berl}(f)$.
\end{definition}
  $\mathrm{Berl}(f)$ is invariant under all permutations of the roots, hence lies in $\F$.  In this case,
discriminants take values in the additive group $\F/\mathcal{P}(\F)$ where, $\mathcal{P}(\F)=\{ c^2+c|c\in \F\}$. $\F/\mathcal{P}(\F)$ classifies separable quadratic extensions of $\F$ just as $\F^{*}/\F^{*2}$ does in $\mathrm{char}(\F)\neq 2.$ 
A fundamental property of Berlekamp discriminant is that the Galois group of 
$f$ over $\F$ contains an odd permutation of the roots of $f$ iff there exists an $\alpha\in \F$ such that $\mathrm{Berl}(f)=\alpha^{2}+\alpha,$ 
 (see [Berlekamp \cite{ber}]). 

\begin{remark}\label{rem_iso}
By Proposition \ref{se_iso}, 
\begin{displaymath}
 \mathrm{Gal}(M_i^{\prime}/K_{n-1})\cong S_d\text{ for } 1\leq i\leq d^{n-1},
\end{displaymath}
and $M_1^{\prime}$,$\dots$,$M_{d^{n-1}}^{\prime}$ are linearly disjoint over $K_{n-1}$.
Additionally, when $q$ is odd, for $1\leq i\neq j\leq d^{n-1}$, the polynomials 
\begin{equation}\label{dsc1}
 \mathrm{Disc}(F(t,\phi)-\beta_i) 
 \text{ and } \mathrm{Disc}(F(t,\phi)-\beta_j)  
\end{equation}
do not have common roots and  are square-free.
Similarly when $q$ is even, for $1\leq i\neq j\leq d^{n-1}$, the polynomials 
\begin{equation}\label{dsc2}
 \mathrm{Berl}(F(t,\phi)-\beta_i) 
 \text{ and } \mathrm{Berl}(F(t,\phi)-\beta_j) 
\end{equation}
  do not have common roots and are square-free.
  Denote by
$E_i^{\prime}$, 
the unique quadratic extension  of $K_{n-1}$ inside  $M_i^{\prime}$ 
given by, \begin{equation}
           \begin{split}
            E_i^{\prime}=K_{n-1}(\sqrt{\mathrm{Disc}(F(t,\phi)-\beta_i)}) \text{ if } q\text{ is odd for } 1\leq i\leq d^{n-1}.\\
              E_i^{\prime}=K_{n-1}(\sqrt{\mathrm{Berl}(F(t,\phi)-\beta_i)}) \text{ if } q\text{ is even for } 1\leq i\leq d^{n-1}.
           \end{split}
           \end{equation}
 $E_i^{\prime}$ is
the fixed field of $A_d$ in $M_i^{\prime}$, where, $A_d$ is the alternating group on $d$ letters, which is a normal subgroup of $S_d$. Clearly, 
$E_{1}^{\prime}$,$\dots$,$ E_{d^{n-1}}^{\prime}$
are linearly disjoint over $K_{n-1}$. This establishes (\ref{sgn}). 
\end{remark}
\begin{proposition}\label{rel-max}
The relative Galois extension $K_{n}/K_{n-1}$ is maximal, that is $\mathrm{Gal}(K_{n}/K_{n-1})\cong 
 S_{d}^{d^{n-1}}$ for all $n\geq 1.$ 
 \end{proposition}
 \begin{proof}
 The proof is immediate from the Proposition \ref{se_iso}. We use induction on $n$. The  result holds true for $n=1$, i.e.,
 \begin{displaymath}
 \mathrm{Gal}(F(t,\phi)/ \overline{\F}_q (\mathbf a))=\mathrm{Gal}(K_1/\overline{\F}_q(\mathbf a))\cong S_d.
\end{displaymath}
Assume the result holds true for $n-1$, i.e.,
\begin{displaymath}
\mathrm{Gal}(K_{n-1}/K_{n-2})\cong S_d^{d^{n-2}}.
 \end{displaymath}
By Proposition \ref{se_iso},
 \begin{equation}\label{bvis}
 K_{n}/K_{n-1}=M_1^{\prime}/K_{n-1}\cdots M_{d^{n-1}}^{\prime}/K_{n-1}. 
 \end{equation}
 and
\begin{displaymath}                                                                 \mathrm{Gal}(M_i^{\prime}/K_{n-1})\cong S_d, \text{ for } 
1\leq i\leq d^{n-1}.
\end{displaymath}
By the result of Lang [\cite{lan}, VI, \S 1.14, \S 1.15],
\begin{displaymath}
 \mathrm{Gal}(K_{n}/K_{n-1})\cong \mathrm{Gal}(M_1^{\prime}/K_{n-1})\times\dots\times(M_{d^{n-1}}^{\prime}/K_{n-1}).
\end{displaymath}
\begin{displaymath}
 \sigma\mapsto (\sigma|_{M_1^{\prime}},\dots,\sigma|_{M_{d^{n-1}}^{\prime}}) \text{ is an isomorphism.}
\end{displaymath}
This gives,
\begin{equation}\label{subprf}
  \mathrm{Gal}(K_{n}/K_{n-1})\cong S_{d}^{d^{n-1}} \text{ for } 1\leq i\leq d^{n-1}.
 \end{equation}
  Thus,  $ \mathrm{Gal}(K_{n}/K_{n-1})$  is maximal for all $n\geq 1$.
 \end{proof}
 We need the following Proposition for completeness. 
\begin{proposition}\label{arb}
Suppose $n$ is any positive integer. Then, any prime $P$ of 
$\overline{\F}_q(\mathbf a)$ ramifying in $K_{n}$ does not ramify in $K_1,\dots,K_{n-1}$. 
 \end{proposition}
 \begin{proof}
  Let us begin by recalling that
  for a fixed $n\in \N$,
  \begin{equation}
   F^{\circ n}(t,\phi)= F(F^{\circ n-1}(t,\phi))\text{ is the }n\text{-th}\text{ iterate  of }F(t,\phi)
  \end{equation}
  and $K_n/\overline{\F}_q(\mathbf a)$ is the splitting field of $F^{\circ n}(t,\phi)$.
  We know  that for any fixed $n\in \N,$ $K_n$ is obtained from $K_{n-1}$ by adjoining the roots of $F(t,\phi)-\beta_i$, for $1\leq i\leq d^{n-1}$, where, $\beta_1$,$\dots$,$\beta_{d^{n-1}}$ are the $d^{n-1}$ distinct roots  of $F^{\circ n-1}(t,\phi)$. We prove the claim by induction.
  For $n=1$, let $\alpha_1,\dots,\alpha_d$ be the $d$ distinct roots of $F(t,\phi)$ in $K_1$. By definition, 
\begin{equation}\label{dpoly}
 F^{\circ 2}(t,\phi)=F(F(t,\phi))=\prod_{i=1}^{d}F(t,\phi)-\alpha_i
\end{equation}
The polynomial $F(t,\phi)$ is irreducible separable over $\overline{\F}_q(\mathbf a)$ which is monic in $t$ of degree $d$ and splits completely in $K_1[t]$.
Since the $d$ polynomials $F(t,\phi)-\alpha_i$ on the right of (\ref{dpoly}) are  linear in $\alpha_i$, they are irreducible separable and pairwise coprime in $K_1[t]$, which splits completely in $K_2[t]$. Therefore for $1\leq i\leq d$, the polynomials  $$F(t,\phi)-\alpha_i \text{ and }F(t,\phi) \text{ do not have any root in common,}  $$ so that, when
 the $\mathrm{char}(\F_q)$ is odd, 
the polynomials
\begin{displaymath}
 \mathrm{Disc}(F(t,\phi)-\alpha_i) \text{ and }\mathrm{Disc}(F(t,\phi)) \text{  do not have common roots, } 
\end{displaymath}
Thus,
\begin{displaymath}
 \mathrm{Disc}(F^{\circ 2}(t,\phi)) \text{ and }
  \mathrm{Disc}(F(t,\phi)) \text{ do not have common roots.}
\end{displaymath}
Similarly when the  $\mathrm{char}(\F_q)$ is even, 
the polynomials 
\begin{displaymath}
 \mathrm{Berl}(F(t,\phi)-\alpha_i) \text{ and } \mathrm{Berl}(F(t,\phi)) \text{ have no common roots,  }
\end{displaymath}
so that,
\begin{displaymath}
 \mathrm{Berl}(F^{\circ 2}(t,\phi)) \text{ and } \mathrm{Berl}(F(t,\phi)) \text{ have no common roots.}
\end{displaymath}
Therefore any prime $P$ of $\F_q(\mathbf a)$  ramifying in $K_2$ does not ramify in 
$K_1$.
Assume the result holds for $n-1$.
As before, let $\beta_1$, $\beta_2$, $\dots$,  $\beta_{d^{n-1}}$ be the $d^{n-1}$ distinct roots of $F^{\circ n-1}(t,\phi)$ in $K_{n-1}.$ Then
\begin{equation}\label{rt}
 F^{\circ n}(t,\phi)=F(F^{\circ n-1}(t,\phi))=\prod_{i=1}^{d^{n-1}}F(t,\phi)-\beta_i
\end{equation}
The polynomial $F^{\circ n-1}(t,\phi)$ splits completely in $K_{n-1}[t]$.
The $d^{n-1}$ polynomials  $F(t,\phi)-\beta_i$ on the right of (\ref{rt})  are monic irreducible separable  and relatively prime in $K_{n-1}[t]$, which split completely in $K_{n}[t]$. Therefore, for $1\leq i\leq d^{n-1},$ the polynomials 
$$F(t,\phi)-\beta_i \text{ and } F^{\circ n-1}(t,\phi) \text{ do not have any root in common},$$
so that, when the $\mathrm{char}(\F_q)$ is odd, the polynomials
\begin{displaymath}
\mathrm{Disc}(F^{\circ n}(t,\phi)) \text{ and } \mathrm{Disc}(F^{\circ n-1}(t,\phi)) \text{ do not have common roots.}
\end{displaymath}
Similarly when the $\mathrm{char}(\F_q)$ is even,
the polynomials
\begin{displaymath}
 \mathrm{Berl}(F^{\circ n-1}(t,\phi))\text{ and } \mathrm{Berl}(F^{\circ n}(t,\phi))\text{ do not have common roots.}
\end{displaymath}
Thus, any prime $P$ of $\F_q(\mathbf a)$
ramifying in $K_n$ does not ramify in $K_{n-1}$.
Equivalently, the splitting fields $K_1$,$K_2$,$\dots$,$K_n$ are linearly disjoint over $\F_q(\mathbf a)$. In short, if the $\mathrm{char}(\F_q)\neq2,$ the polynomials
$\mathrm{Disc}(F(t,\phi))$, $\mathrm{Disc}(F^{\circ 2}(t,\phi))$, $\dots$, $\mathrm{Disc}(F^{\circ n}(t,\phi))$ are relatively prime in $\overline{\F}_q(\mathbf a)$. Therefore,
\begin{equation}
\prod_{i=1}^{n}\mathrm{Disc}(F^{\circ i}(t,\phi)) \text{ is not a square.}  
\end{equation}   
Similarly,   if the $\mathrm{char}(\F_q)=2$, the polynomials
$\mathrm{Berl}(F(t,\phi))$,$\dots$,$\mathrm{Berl}(F^{\circ n}(t,\phi))$ are relatively prime in  $\overline{\F}_q(\mathbf a)$. Therefore,
\begin{equation}\prod_{i=1}^{n}\mathrm{Berl}(F^{\circ i}(t,f)) \text{ is not a square. }\end{equation}  
\end{proof}
Now we are set to prove parts (2) and (3) of Theorem
 \ref{main_th1}.
 We prove the claim by induction on $n$. The result is  true for $n=1$. Assume that, the result holds true for $n-1$, i.e.,
 \begin{displaymath}
 \mathrm{Gal}(F^{\circ n-1}(t,\phi)/\overline{\F}_q(\mathbf a))\cong [S_d]^{n-1}.
 \end{displaymath}
 Accordingly,
 \begin{displaymath}
[K_{n-1}:\overline{\F}_q(\mathbf a)]=|[S_d]^{n-1}|=|\underbrace{S_d\wr\dots\wr S_d}_\text{$(n-1)$ times}|=|(\underbrace{S_d\wr\dots\wr S_d)}_\text{$(n-2)$ times}|^{d}\cdot d!
\end{displaymath}
Therefore,
\begin{displaymath}
[K_{n-1}:\overline{\F}_q(\mathbf a)]=|[S_d]^{n-1}| =(d!)^{d^{n-2}}\cdot(d!)^{d^{n-3}}\cdots (d!).
\end{displaymath}
 From (\ref{subprf}), we know that $K_n$ has degree $(d!)^{d^{n-1}}$ over $K_{n-1}$,
 and
\begin{displaymath}
[K_{n}:\overline{\F}_q( \mathbf a)]= [K_{n}:K_{n-1}]\cdot[K_{n-1}:K_{n-2}]\cdots[K_1:\overline{\F}_q(\mathbf a)].
\end{displaymath}
\begin{displaymath}
= [K_{n}:K_{n-1}]\cdot[K_{n-1}:\overline{\F}_q(\mathbf a)]
\end{displaymath}
Therefore,
 \begin{equation}
 [K_n:\overline{\F}_q(\mathbf a)]= (d!)^{d^{n-1}}\cdot|[S_d]^{n-1}| =|[S_d]^{n}|.
\end{equation}
  Thus
 \begin{equation}\label{linear-gal}
 \mathrm{Gal}(K_{n}/\overline{\F}_q(\mathbf a))\cong [S_d]^{n} \text{ for all } n\geq1.
\end{equation} 
Thus,  $\mathrm{Gal}(F^{\circ n}(t,f)/\F_q(\mathbf a))$ is maximal for $n\geq 1$. This implies
the arboreal representation associated to $F(t,\phi)$ is surjective.  
This  concludes the conditions  (2) and (3) of  Theorem \ref{main_th1}.
\section{Applications}
\label{appl}
In this section we outline the proof of  Corollary \ref{stability}.
The objective is to
count the number of polynomials $f\in \F_q[t]$ of degree $r$ for which $F(t,f)$, $F^{\circ 2}(t,f)$,$\dots$,$F^{\circ l}(t,f)$ are simultaneously square-free in  $\F_q[t]$ for $l>1$.
\subsection{Independence of cycle structure}\label{indcy}
Let $ \mathcal{M}_n\subset \F_q[t]$ be the collection of all monic polynomials of degree $n$ over $\F_q$ and 
 $f\in \mathcal{M}_n$ be a separable polynomial of degree $n>0$.
 We say that its cycle structure is $\lambda(f)=(\lambda_1,\dots,\lambda_n)$, if in the decomposition $f=\prod_{i}P_i$ into prime factorization, we have $\#\{i:\mathrm{deg}(P_i)=j\}=\lambda_j$. Thus we get a partition of $n$. 
 The cycle structure of a permutation 
 $\sigma$ of $n$ letters is the partition
 $\lambda(\sigma)= (\lambda_1,\dots,\lambda_n)$ of $n$, if in the decomposition of $\sigma$ as a product of disjoint cycles, there are $\lambda_j$ cycles of length $j$ and $\sum{j}\lambda_j=n$.
 It is well known that the distribution over $\mathcal{M}_n$ of factorization types tends to the distribution of cycle types in $S_n$.
By
the cycle type of  $f(t)\in\mathcal{M}_n $, we mean the partition of $n$ obtained from the 
degrees of the irreducible factors of $f$, which we identify with the associated conjugacy class of
the symmetric group $S_n$ on $n$ letters.
   For each partition $\lambda\vdash n$, we denote by $p(\lambda)$, the probability that a random permutation on $n$ letters has cycle structure $\lambda$:
  \begin{equation}\label{cauchy}
  p(\lambda)= \frac{\#\{\sigma\in S_n: \lambda({\sigma})=\lambda\}}{\# S_n}                                                                                                                                        
  \end{equation}
 
 When $q$ is large $p(\lambda)$ is  asymptotic to the  probability that
a random polynomial $f\in \mathcal{M}_n$ has cycle structure $\lambda$ (see Cohen \cite{sdc}). 
Additionally Cohen [\cite{sdc3}, Theorem 3] observed, cycle structures (factorization types) of distinct separable  polynomials $f_1,\dots,f_l\in \F_q[t]$  become linearly independent as $q\rightarrow \infty.$ He   
  also proved that
the  Chebotarev density
theorem and the Weil bound can be coupled to count  not only the occurrence of irreducibles but also the occurrence of polynomials in 
appropriate sequences with any prescribed cycle type. This is generalized in  [\cite{pol1}, Theorem 2]  and [\cite{abz}, Theorem 3.1].
 For a partition $\lambda\vdash n$, let $\tau(\lambda)$ be the  characteristic function of $f\in \mathcal{M}_n$ of cycle structure $\lambda:$
\begin{displaymath}\tau_\lambda(f)=
\begin{cases}
1, \text{ if } \lambda({f})=\lambda\\
0, \text{ otherwise}
\end{cases}\end{displaymath}
 We state the following result on the   independence of cycle structure of non-associate, separable polynomials in $\mathcal{M}_n$.
\begin{theorem} [Theorem 1.4, \cite{abz}]\label{ab-z}
 For fixed positive integers $n$ and $k$\begin{displaymath}
 \frac{1}{q^{n}}\sum_{f_i\in \mathcal{M}_n}\tau_{\lambda}(f_1)\cdots\tau_{\lambda}(f_k)=\underbrace{p(\lambda)\cdots p(\lambda)}_\text{$k$ times}+O(q^{-1/2})
\end{displaymath}
uniformly for all distinct, monic separable polynomials $f_1,\dots,f_k\in \mathcal{M}_n$ and on all partitions $\lambda\vdash n$ as $q\rightarrow\infty$.
\end{theorem}
 The 
 statistics of Theorem \ref{ab-z} is induced from the statistics of
cycles structure of tuples of elements in the direct product $(S_n)^{k}$
of $k$ copies of
the symmetric group $S_n$ on $n$ letters,
which is due to a more general explicit Chebotarev
theorem for function fields, [\cite{abz}, Theorem 3.1].
 
Corollary \ref{stability}  reduces to
count the 
number of monic polynomials $f$ in $\F_q[t]$ of degree $r$ for which $\mu(F^{\circ i}(t,f))=\pm1$ for $1\leq i\leq l$, where $\mu$ is the Mobius function for $\F_q[t]$.
For $F\in \F_q[t]$ the M\"{o}bius function $\mu(F)$ can be computed in terms of discriminant of $F$.
Therefore, in the odd characteristic case, by the work of  Conrad [\cite{ccg},  Lemma  4.1],
\begin{equation}\label{disc1}
 \mu(F)=(-1)^{\mathrm{deg}F}\chi_2(\mathrm{Disc}(F)).
\end{equation}
Similarly, in the even characteristic case,
by the work of  Carmon [\cite{car}, \S 2.3],
\begin{equation}\label{disc2}
 \mu(F)=(-1)^{\mathrm{deg}F}\chi_2(\mathrm{Berl}(F)).
\end{equation}
where $\chi_2$ is the quadratic character of $\F_q$.
By Proposition \ref{arb},
the discriminant of the polynomials $F^{\circ 1}(t,\phi)$,$\dots$, $F^{\circ l}(t,\phi)$ are linearly independent in the $\F_2$ vector space $\F_q(\mathbf a)^{\times}/(\F_q(\mathbf a)^{\times})^{2}$.
 Therefore,
for suitable, irreducible specializations $f\in \F_q[t]$, the polynomials  $F^{\circ 1}(t,f)$, $\dots$, $F^{\circ l}(t,f)$ are separable and distinct, hence their cycle structure become independent as $q\rightarrow \infty$. 
Now applying Theorem \ref{ab-z}, when 
$\tau_{\lambda}=\mu$,
 the number of $f\in \F_q[t]$ of degree $r$ for which $F(t,f)$ is stable over $\F_q$, such that, $F(t,f)$  (which has the adjusted critical orbit of  length $l>1$),  is 
 $$ \frac{q^{r}}{2^{l}}+O_{l,\mathrm{deg}F^{\circ i}}(q^{r{-1}/{2}}),\quad q\rightarrow \infty.
 $$
the implicit constant in the $O$-notation depends only on $l$ and the degrees $\mathrm{deg}_tF^{\circ i}$ for $1\leq i\leq l.$
\section{Acknowledgment}
 I thank Prof. Soumya Das, Department of Mathematics, IISc,  for the useful discussions  during the early stage of this work.
 I am grateful to Prof. Ila Varma, Department of Mathematics,  University of Toronto and Prof. Aditya Karnataki,  Chennai Mathematical Institute for reading carefully through the earlier drafts of this manuscript and for their valuable remarks which improved this  article to a great extent.
 I thank the Department of Science and Technology, Government of India, for the financial help under the Women Scientist Scheme Project.
 
 \bibliographystyle{plain}

\end{document}